\documentclass{amsart}
\usepackage{amssymb}
\usepackage{amsmath}
\usepackage{amsfonts}
\usepackage{pstricks}
\usepackage{pst-node}
\usepackage{graphicx}
\theoremstyle{plain}
\newtheorem{thm}{Theorem}[section]

\newtheorem{dfn}[thm]{Definition}
\newtheorem{rmk}[thm]{Remark}

\newtheorem{prb}[thm]{Problem}

\theoremstyle{remark}

\def\pmc#1{\setbox0=\hbox{#1}
    \kern-.1em\copy0\kern-\wd0
    \kern.1em\copy0\kern-\wd0}

\begin{document}

\bigskip

\title[On
the Alexandroff-Borsuk  problem]{
On the Alexandroff-Borsuk  problem  }

%\centerline{\bf \tt This is the version of
%\today. Last changes by Du\v san}

\bigskip

\author[M.~Cencelj]{Matija Cencelj}
\address{Faculty of Education and
Institute of Mathematics,
Physics and Mechanics,
University of Ljubljana, Kardeljeva plo\v s\v cad 16, Ljubljana
1000,
Slovenia} \email{matija.cencelj@guest.arnes.si}

\author[U.~H.~Karimov]{Umed H. Karimov}
\address{Institute of Mathematics,
Academy of Sciences of Tajikistan, Ul. Ainy $299^A$,
Dushanbe 734063,
Tajikistan}
\email{umedkarimov@gmail.com}

\author[D.~D.~Repov\v{s}]{Du\v{s}an D. Repov\v{s}}
\address{Faculty of Education and Faculty of Mathematics and Physics,
University of Ljubljana, Kardeljeva plo\v s\v cad 16,
 Ljubljana
1000,
Slovenia} \email{dusan.repovs@guest.arnes.si}

\subjclass[2010]{Primary: 54F15, 55N15; Secondary: 54G20, 57M05}
\keywords{ANR, finite polyhedron, homotopy equivalence,
$\varepsilon$-map, cellular map, almost-smooth manifold, $|E_8|$-manifold,
Kirby-Siebenmann class, Galewski-Stern obstruction, non-triangulable manifold, Alexandroff-Borsuk Manifold Problem}

\begin{abstract}
We investigate the classical Alexandroff-Borsuk  problem in the
category of non-triangulable manifolds: Given an $n$-dimensional
compact non-triangulable manifold $M^n$ and $\varepsilon > 0$,
does there exist an $\varepsilon$-map of $M^n$ onto an
$n$-dimensional finite polyhedron which induces a homotopy
equivalence?
\end{abstract}

\date{\today}

\maketitle

\section{Introduction}

 In 1928 Alexandroff \cite{A} proved the following important theorem:

{\bf Alexandroff Theorem.}
{\it Every
$n$-dimensional
compact metric space $X$ has the following properties:
\begin{itemize}
\item
for every
$\varepsilon > 0,$
$X$
admits  an $\varepsilon$-map $f:X^n \to P^n$
onto some $n$-dimensional
finite polyhedron $P^n$; and
\item
for some $\mu > 0,$
$X$ does not admit any $\mu$-map
$g:X^n \to Q^k$
of $X^n$ 
 onto a polyhedron $Q^k$
of  dimension $k<n$.
\end{itemize}}

Recall that for a metric space $X$ and $\varepsilon >0$, 
a continuous
map $f\colon  X \to P$  is called
an
$\varepsilon$-map if the
preimage $f^{-1}(p)$ of every point $ p \in P$ has diameter $< \varepsilon.$

The condition of compactness in Alexandroff's theorem above
 is essential since
in 1953
 Sitnikov
\cite{E, Sit} constructed 
an example of a $2$-dimensional subspace of $R^3$ which
can be $\varepsilon$-mapped  onto a 1-dimensional polyhedron for arbitrarily small
$\varepsilon >0$.

In 1954 Borsuk \cite{B} asked whether every compact absolute neighborhood retract (ANR) is homotopy equivalent
to a finite polyhedron.
This  difficult
question was answered in the affirmative
 in 1977  by
West \cite{W}.

It follows by Wall's obstruction theory \cite{Wa}
 that for every $n > 2$,
every $n$-dimensional
compact ANR  is homotopy equivalent to an
$n$-dimensional polyhedron, having the structure of a finite simplicial
complex (cf. \cite{W}). 

The following natural 
problem has been opened for a very long time (compare with \cite{KR}):

\medskip

{\bf Alexandroff-Borsuk ANR Problem.}
{\it Given any compact $n$-dimensional absolute neighborhood
retract  $X^n$ and any $\varepsilon > 0$, does there exist an
$\varepsilon$-map $f:X^n \to P^n$ of $X^n$ onto a finite $n$-dimensional
polyhedron $P^n$ which is a homotopy equivalence?}

\medskip

In this paper we shall consider the Alexandroff-Borsuk problem for the
category of non-triangulable manifolds. Since every
topological
 $n$-manifold is a
separable metric locally Euclidean space  of dimension $n$
(cf. \cite{L}), it is a locally contractible finite-dimensional space and
therefore it is an ANR (cf. \cite{Bor}).

 It follows
by the
West Theorem  mentioned
above
that every compact  manifold has the
homotopy type of a finite polyhedron. This fact was first proved
in 1969
 by Kirby and Siebenmann  \cite{KS}. So we have the
following natural special case of the Alexandroff-Borsuk problem:

\medskip

{\bf Alexandroff-Borsuk Manifold Problem.} 
{\it Given a
 compact $n$-di\-men\-sio\-nal manifold $M^n$ and $\varepsilon>0,$ does there
exist an $\varepsilon$-map of $M^n$
onto a
finite $n$-dimensional polyhedron $P^n $ which is a homotopy
equivalence?}

Recall that for every $n \geq 4,$ there exists a closed $n$-dimensional manifold  which
is not a polyhedron. Such manifolds were first constructed
by Freedman \cite{F}
 in
dimension 4, by
Galewski and Stern \cite{GS}
in  dimension 5,
 and by
Manolescu \cite{M}
 in  dimensions $n \geq 6$.

\medskip

The following is the main result of our paper:

\medskip

{\bf Main Theorem.}
{\it The Alexandroff-Borsuk Manifold Problem has an affirmative
 solution for the
non-triangulable closed manifolds of Freedman, Galewski and Stern, and
Manolescu.}

\medskip

\section{Preliminaries}

We shall work with the categories of separable metric spaces, CW
complexes and continuous maps. In these categories all classical
definitions of dimension coincide: ${\rm dim}\ X = {\rm ind}\ X
={\rm Ind}\ X$ (cf. \cite{E}).

We list some definitions and theorems which we shall need in the sequel:

\begin{thm}\label{Cellular} {\rm(Cellular Approximation
Theorem \cite[p. 77]{Wh})}. Let $(X,A)$ and $(Y,B)$ be relative
CW complexes and let $f\colon (X,A) \rightarrow (Y,B)$ be a
continuous map. Then $f$ is homotopic $(rel \ A)$ to a cellular
map.
\end{thm}

Recall that a map $f\colon X \rightarrow Y$ between CW complexes $X$ and $Y$ 
is said to be {\it cellular} if
$f(X_n)\subset Y_n$ for every $n$.
 The Simplicial Approximation
Theorem is a special case of the Cellular Approximation Theorem,
cf. \cite[p. 76]{Wh} for details.

\begin{dfn} \label{Almost-smooth} A manifold $M$
is said to be almost-smooth if
it admits  a smooth structure in the complement $M\setminus \{p\}$ of any
point $p\in M$.
\end{dfn}

\begin{thm}\label{Cairns} {\rm (Cairns-Whitehead } \cite{C2, C1, Jam, JWh}{\rm ).}
Every smooth manifold can be given a simplicial structure.
\end{thm}

Freedman has proved the following important theorem (cf. \cite[Corollary
1.6]{F}):

\begin{thm}\label{Freedman}
There is a closed connected  almost-smooth
 4-manifold $|E_8|$
with the intersection matrix $E_8$.
\end{thm}

Freedman also established the following surprising fact:

\begin{thm}\label{Freedman`s Corollary 1.6}
Either $| E_8|$  is the first example in any dimension of a
manifold which is not homeomorphic to a polyhedron, or the 3-dimensional
Poincar\'e conjecture is false.
\end{thm}

Since Perel'man has proved the Poincar\'e conjecture (cf. e.g. \cite{MT}), it follows that
 the 4-manifold  $|E_8|$ is not homeomorphic to any polyhedron. However, the
complement of any  point in $|E_8|$
 can be given a polyhedral
structure, since it admits a smooth structure on the complement of any point 
(cf. \cite[Theorem on p. 116]{FQ})
and by the Cairns-Whitehead Theorem \ref{Cairns}
every smooth manifold is triangulable.

We briefly recall the Kirby-Siebenmann and the Galewski-Stern
obstructions and some of their theorems (cf. \cite{GS2, Ru}).

\begin{dfn}{\rm (cf.} \cite[pp. xii-xiii, 78]{Ru}{\rm )}. $BTOP$ is the classifying space for stable topological
bundles and
 $\chi$ is some element of the $4$-dimensional cohomology
group $H^4(BTOP; \mathbb{Z}_2)$ which is called the universal
Kirby-Sieben\-mann class.
\end{dfn}

The construction of the CW complex $BTOP$ and  the element
$\chi$ is given e.g. in \cite[pp. xii-xiii, 78]{Ru}.

\begin{dfn} {\rm(cf. e.g.} \cite[pp. 403-404]{FW}{\rm )}.
Let $M^n$ be a topological manifold. Then the
 topological tangent bundle
$\tau_M^n$ of $M^n$ is a neighborhood $U$ of the diagonal $\Delta
\subset M^n \times M^n$ such that the projection $p_1\colon U
\rightarrow M^n$ is a topological $\mathbb{R}^n-bundle.$ Here,
$p_1(x,y) = x.$
\end{dfn}

Let $M$ be a topological manifold, and let $f\colon M \rightarrow
BTOP$ classify the stable tangent bundle of $M$. The class
$$f^4(\chi) \in H^4(M; \mathbb{Z}_2)$$ is a well-defined invariant
of $M$ since $f$ is unique up to homotopy. The Kirby-Siebenmann
class $\Delta (M)$ of $M$ is by definition the element
$f^4(\chi)$.

If $U \subset M,$  $i\colon U \rightarrow M$
is the inclusion, 
and $f_M$, $f_U$ classify the corresponding stable tangent bundles
then $f_M \circ i \simeq f_U$ and we have $i^4(\Delta (M)) =
\Delta (U)$.

\begin{thm}\label{KS} {\rm (cf. e.g.} \cite[p. 78]{Ru}{\rm )}. Let
$M$ be a topological manifold.
If $M$ admits a PL structure, in particular if it admits a smooth
structure, then $\Delta (M) = 0$. For $dim\ M \geq 5$  the
converse
also
 holds: if $\Delta (M) = 0,$ then $M$ admits a $PL$
structure.
\end{thm}

\begin{thm}\label{GS} {\rm (cf.} \cite{DFL, GS2, S}{\rm )}. Let $M^n$ be a topological
manifold, $n \geq 5$. Then the
 Galewski-Stern obstruction for a manifold $M^n$ to having
 a
simplicial triangulation is the image
$$\beta(\Delta(M^n))\in
H^5(M^n; \mbox{\rm Ker}\ \mu)$$ of the
 Kirby-Siebenmann class $\Delta(M^n)$ by
  the
Bockstein homomorphism $\beta$ for the exact sequence of
coefficients:
$$0 \longrightarrow\mbox{\rm Ker}\ \mu \longrightarrow \theta^H_3 \stackrel  {\mu}{ \longrightarrow}\mathbb{Z}_2 \longrightarrow 0.$$
\end{thm}

The 
homomorphism $\mu: \theta^H_3 \to \mathbb{Z}_2 $ is the Rokhlin
invariant homomorphism of the abelian group $\theta^H_3$  of
homology cobordism classes of oriented PL homology $3$-spheres
with the operation of connected sum, cf. e.g. \cite{DFL, S}. For
the purposes of our paper it suffices to know that $\mbox{Ker}\
\mu$ is some nontrivial abelian group.

\begin{thm}\label{Ferry} {\rm (cf. }
\cite[for $n = 3$]{J},
  \cite[for $n = 4$]{FW}, {\rm and}
 \cite[for $n \ge 5$]{CF}{\rm )}. Let $\alpha$ be an
open cover of an $n$-manifold $M.$ Then there exists an open cover
$\beta$ of $M$  such that any proper $\beta$-map $g\colon M \to N$
onto any $n$-manifold $N$ is homotopic through  $\alpha$-maps
to a homeomorphism.
\end{thm}

For uniformity we denote the manifolds of Freedman,
Galewski and Stern, and Manolescu by
 $M_F^4,$ $M_{GS}^5$ and
$M_M^{5+n}$, respectively.

\medskip

\section{Proof of the Main Theorem}

First consider the manifold $M_F^4.$ 
According to Theorems
\ref{Cairns} 
and 
\ref{Freedman}, 
the complement 
$M_F^4
\setminus \{p\}$
of
 a point  
$p \in M_F^4$
is a polyhedron. Let $\varepsilon>0$ be any positive
number. 
Let $U \subset M_F^4$ be an open topological ball in $M_F^4$ with center at
$p$ and with diameter less than $\varepsilon.$ Consider a
triangulation $T$ of $M_F^4 \setminus \{p\}$ and let $K$ be the
polyhedron which is the union of all simplices of $T$ which intersect
with $M_F^4 \setminus U$ and let $L$ be the compactum which is the union of
all the remaining simplices in $U$ and the point $p$.

Obviously, $L$ is a locally
contractible $4$-dimensional space and therefore by the Borsuk theorem
\cite{Bor} it is a compact ANR. By the above mentioned theorems of West and Wall
 it follows that $L$ is homotopy equivalent to a finite
$4$-dimensional  polyhedron, call it $B$. We have a homotopy
equivalence $f\colon L \rightarrow B$. Consider the closed star
$st(f(p))$ of the point $f(p)$ in the finite polyhedron $B$, that
is the union of all simplices of $B$ containing the point $f(p)$.
This is a closed neighborhood of the point $p$. Consider the
preimage of $st(f(p)).$  We get a closed neighborhood of the point
$p$ in $L$.

Consider the second barycentric subdivision of  $st(f(p))$ and the
closed star of the point $f(p)$ in it. Call it $st_2(f(p)).$
Obviously,
$$f^{-1}(st_2(f(p))) \subset {\rm Int}\ f^{-1}(st(f(p))).$$
Let $d$ be the minimal distance between the points
of the
boundary $\partial(f^{-1}(st(f(p)))$ and the points of
compactum $f^{-1}(st_2(f(p))).$ Clearly, $d > 0.$

Consider small
triangulation of  $T$ such that the diameters of all of its simplices
are less than $d$ and let $Q$ be the union of all simplices of this
triangulation which intersect  $f^{-1}(st_2(f(p)))$ and the
point $p.$ We get a 
relative CW complex $(L, Q)$ and the
mapping $$f\colon (L, Q) \rightarrow (B, st(f(p))).$$ By the Cellular
Approximation Theorem \ref{Cellular},
 the map $f$ is homotopy
equivalent to a map $g$  such that its restriction to $L\cap K$ is
a simplicial map (in our case the relative CW complexes are
obviously relative simplicial complexes).

Since $M_F^4$ and $L$ are ANR's, the pair $(M_F^4, L)$ has the
homotopy extension property with respect to any space \cite[p.
120]{Hu}, \cite{Wh}. Therefore $M_F^4$ is homotopy equivalent to
the quotient space $M_F^4\cup_g B$, cf. \cite[p. 26, Corollary
(5.12)]{Wh}.

The space $M_F^4\cup_g B$ is the union
of two finite polyhedra intersection of which is $g(L\cap K)$,
a
subpolyhedron of both of these polyhedra and therefore $M_F^4\cup_g
B$ is a finite polyhedron. Obviously, the quotient map of $M_F^4$
onto $M_F^4\cup_g B$ is an $\varepsilon$-map.

Consider now the Galewski-Stern manifold $M_{GS}^5$. The
obstruction to triangulability of any manifold $M^n$ for $n\geq5$
is the obstruction $$\beta(\Delta)(M^n) \in H^5(M^n;
\mbox{Ker}\ \mu)$$
(cf. Theorem \ref{GS}). Manolescu  \cite{M}
has
showed that
$\beta(\Delta)(M^5_{GS})$ is nontrivial  and therefore the
manifold $M_{GS}^5$ is non-triangulable.

However, for any connected
closed $5$-dimensional manifold,
$$H^5(M^5\setminus \{p\}; \mbox{\rm Ker}\ \mu) = 0$$ and therefore the Galewski-Stern
obstruction $\beta(\Delta)(M^5_{GS}\setminus \{p\})$ is trivial
and $M^5_{GS}\setminus \{p\}$ is an infinite polyhedron. Using the
arguments similar to those used in the proof of the theorem for
the Freedman manifold it follows that for every $\varepsilon > 0$
there exists an $\varepsilon$-map of $M_{GS}^5$ onto a
$5$-dimensional polyhedron $P^5$ which is a homotopy equivalence.

The Manolescu non-triangulable manifold $M_M^{5+n}$ is the product
of $M_{GS}^5$ with the torus $T^n$. Since  $M_{GS}^5$
admits an $\varepsilon$-map onto $P^5$ which is a homotopy
equivalence for an arbitrarily small $\varepsilon>0$,  it follows that obviously, $$M_M^{5+n} = M_{GS}^5 \times
T^n$$ also admits an $\varepsilon$-map which is a homotopy equivalence onto
$P^5\times T^n$ for arbitrarily small $\varepsilon>0$.

\section{Some complementary results and remarks}

\begin{thm}\label{Noalmost-smooth}
There do not exist any  non-triangulable almost-smooth manifolds $M^n$ for any
$n
> 4$.
\end{thm}

\begin{proof}
Since $M^n$ is non-triangulable it does not have a PL structure
and therefore the Kirby-Siebenmann obstruction $$\Delta(M^n) \in
H^4(M^n;Z_2)$$ is nonzero, by Theorem \ref{KS}. It follows from the
exact sequence of the pair $(M^n, M^n\setminus \{p\}):$
$$0=H^4(M^n, M^n\setminus \{p\}; \mathbb{Z}_2)\rightarrow H^4(M^n;\mathbb{Z}_2)
\rightarrow H^4(M^n\setminus \{p\}; \mathbb{Z}_2)$$ \noindent
that $\Delta ( M^n\setminus \{p\})$ is nonzero for $n > 4$.
Therefore $ M^n\setminus \{p\}$ does not have a PL structure and
is thus not smooth.
\end{proof}

\begin{thm}\label{Non-triangulable}
A non-triangulable connected manifold of dimension $n > 4$ has an
infinite simplicial complex structure in the complement of a point
if and only if $n=5.$
\end{thm}

\begin{proof}
The obstruction class $\beta(\Delta)(M\setminus \{p\})$ is
obtained as the image of $\beta(\Delta)(M)$ by the restriction
$$H^5(M; \mbox{\rm Ker}\ \mu )
\rightarrow 
H^5(M \setminus \{p\}; \mbox{\rm Ker}\
\mu).$$ When $n=5$, we have $$H^5(M \setminus \{p\}; \mbox{\rm
Ker}\ \mu)=0$$ since $M$ is a connected manifold, hence
 the obstruction vanishes. When $n >
5$, we see that by the exact sequence of the pair $(M, M\setminus \{p\})$ we get

$$0=H^5(M, M\setminus \{p\}; \mbox{\rm Ker}\ \mu)\rightarrow H^5(M;\mbox{\rm Ker}\ \mu) \rightarrow H^5(M\setminus \{p\};
\mbox{\rm Ker}\ \mu).$$ 
The restriction is injective, so the image
of $\beta(\Delta)(M)$ is non-zero and $M\setminus \{p\}$ is
non-triangulable as simplicial complex, by Theorem \ref{GS}.
\end{proof}

\begin{rmk}
For non-triangulable manifolds there do not exist
$\varepsilon$-maps onto a triangulable manifold for small enough
$\varepsilon >0$. This follows from deep results of Chapman and Ferry,
Ferry and Weinberger, and Jakobsche {\rm (cf. Theorem \ref{Ferry})}.
\end{rmk}

\section{Epilogue}

The non-triangulable closed $4$-manifolds of Freedman \cite{F,
Sc}, $5$-manifolds of Galewski and Stern \cite{DFL,GS}, and
$n$-manifolds
of Manolescu  for $n\geq 6$ \cite[p. 148]{M} have nice geometric
descriptions. All of them are homotopy equivalent to  polyhedra
of the corresponding dimension. We have polyhedral homotopy
representatives $PH^4_{F}, PH^5_{GS}, PH^{5+n}_M$ of the above
mentioned non-triangulable manifolds.

\begin{prb}
Find a geometric description of the polyhedra $PH^4_{F},
PH^5_{GS}$ and $PH^{5+n}_M$.
\end{prb}

The Alexandroff-Borsuk problem is solved for the special class of
non-triangulable manifolds and is still open for general
non-triangulable manifolds. The following problems seem to be of
interest:

\begin{prb}
Let $M^n$ be any non-triangulable manifold. Does there exist
 any polyhedron $P$ embeddable
 in $M^n$, such that $M^n
\setminus P$ is also a polyhedron?
\end{prb}

According to our Main Theorem, for every positive number
$\varepsilon$ and for the manifolds of Freedman, Galewski and Stern,
and Manolescu there exist $\varepsilon$-maps of these manifolds
onto some polyhedra $PH^4_{F}, PH^5_{GS}, PH^{5+n}_M,$
respectively. The following version of the Alexandroff-Borsuk
Manifold Problem remains open:

\begin{prb}
Does there exist
for every compact $n$-dimensional manifold  $M^n,$ a finite $n$-dimensional polyhedron $P^n $ such that for an
arbitrarily small $\varepsilon>0$ there exists an $\varepsilon$-map $f\colon
M^n \to P^n $ which is a homotopy equivalence?
\end{prb}

The answer for the corresponding version of
Alexandroff-Borsuk ANR Problem
is negative, even for
$1$-dimensional compact absolute retracts (AR), i.e. for the dendrites.

\section{Acknowledgements}
This research was supported by the Slovenian Research Agency grants P1-0292, J1-5435, J1-6721
and J1-7025.
We are very grateful to A.V. Chernavskii
for reference \cite{CF}, to  S. Ferry for references \cite{FW, J},
and
C. Manolescu for sketches of proofs of Theorems
\ref{Noalmost-smooth} and
 \ref{Non-triangulable}.
We also acknowledge
their advice during the preparation of this paper.
We thank the referee for comments and suggestions.

\end{document}